\documentclass{amsart}
\usepackage[all]{xy}

\usepackage{mathrsfs}
\usepackage{amssymb}

\newtheorem{theorem}{Theorem}[section]
\newtheorem{lemma}[theorem]{Lemma}
\newtheorem{proposition}[theorem]{Proposition}
\newtheorem{corollary}[theorem]{Corollary}

\theoremstyle{definition}

\theoremstyle{remark}
\newtheorem{remark}[theorem]{Remark}

\numberwithin{equation}{section}

\begin{document}

\title[Pluri-canonical Maps in Positive Characteristic]
{Pluri-canonical Maps of Varieties of Maximal Albanese Dimension in
Positive Characteristic}

\author{Yuchen Zhang}
\address{Department of Mathematics, University of Utah, Salt Lake City, Utah 84102}
\curraddr{155 S 1400 E Room 233, Salt Lake City, Utah 84112}
\email{yzhang@math.utah.edu}



\date{\today}



\begin{abstract}
We show that if $X$ is a nonsingular projective variety of general
type over an algebraically closed field $k$ of positive
characteristic and $X$ has maximal Albanese dimension and the
Albanese map is separable, then $|4K_X|$ induces a birational map.
\end{abstract}

\maketitle



\section{Introduction}

Let $X$ be a nonsingular projective variety of general type over an
algebraically closed field $k$, $n=\dim X$ and $K_X$ be a canonical
divisor. Since $K_X$ is big, for any sufficiently large positive
integer $m$, the linear series $|mK_X|$ induces a birational map. It
is an important problem to bound this integer $m$. For
$\text{char}\; k=0$, by a result of \cite{HM}, \cite{Ta}, \cite{T06}
and \cite{T07}, there exists a non-effective bound for $m$ which
only depends on the dimension $n$. When $X$ has maximal Albanese
dimension, \cite{CH} and \cite{JLT} show the optimal result that
$|3K_X|$ is birational. Furthermore, \cite{CH} shows that if the
Albanese dimension is $n-1$, then $|6K_X|$ is birational, and if the
Albanese dimension is $n-2$, then $|7K_X|$ is birational.

In this paper, we will generalize these results to positive
characteristic.

\begin{theorem}
(See Theorem \ref{main}) Let $X$ be a smooth projective variety of
general type over an algebraically closed field $k$ of
characteristic $p>0$. If $X$ has maximal Albanese dimension and the
Albanese map is separable, then $|4K_X|$ induces a birational map.
\end{theorem}

Our strategy is similar to that in \cite{CH} where the Fourier-Mukai
transform (Lemma \ref{hacon}) is used repeatedly to produce sections
of $mK_X$. Their approach uses multiplier ideals and
Kawamata-Viehweg vanishing in an essential way. In positive
characteristic, the theory of Fourier-Mukai transforms still
applies, and multiplier ideals can be replaced by test ideals.
However, Kawamata-Viehweg vanishing is known to fail. Inspired by
\cite{H}, \cite{M} and \cite{S}, we replace Kawamata-Viehweg
vanishing by the Frobenius map and Serre vanishing. Combining this
with the Fourier-Mukai transform, we obtain that $|4K_X|$ is birational.
It seems that new ideas are required to investigate the third
pluri-canonical map.

\subsection*{Ackowledgements} The author would like to thank his advisor
Professor Christopher Hacon for suggesting this problem and many
useful discussions. The author would also like to thank the referee
for pointing out several mistakes in the previous version of this
paper.

\section{Asymptotic Test Ideals}

Suppose that $X$ is a smooth $n$-dimensional variety over an
algebraically closed field $k$ of characteristic $p>0$. Let
$\omega_X$ denote the canonical line bundle on $X$. We denote
$F:X\rightarrow X$ the absolute Frobenius morphism, that is given by
the identity on the topological space, and by taking the $p$-th
power on regular functions. Let $\text{Tr}:F_*\omega_X\rightarrow
\omega_X$ be the trace map and $\text{Tr}^e:F_*^e
\omega_X\rightarrow \omega_X$ be the $e$-th iteration of the trace
map.

We follow the definitions given in \cite{M}. For other equivalent
definitions, see \cite{BMS} and \cite{S}. Given a nonzero ideal
$\mathfrak{a}$ in $\mathcal{O}_X$, the image
$\text{Tr}^e(\mathfrak{a}\cdot\omega_X)$ can be written as
$\mathfrak{a}^{[1/p^e]}\cdot\omega_X$ for some ideal
$\mathfrak{a}^{[1/p^e]}$ in $\mathcal{O}_X$. Given a positive real
number $\lambda$, one can show that
$$\left(\mathfrak{a}^{\lceil\lambda p^e\rceil}\right)^{[1/p^e]}
\subseteq\left(\mathfrak{a}^{\lceil\lambda
p^{e+1}\rceil}\right)^{[1/p^{e+1}]}$$ for every $e\geqslant 1$ where
$\lceil t\rceil$ means the smallest integer $\geqslant t$. Hence,
there is an ideal $\tau(\mathfrak{a}^\lambda)$, called the
\textbf{test ideal} of $\mathfrak{a}$ of exponent $\lambda$, that is
equal to $\left(\mathfrak{a}^{\lceil \lambda
p^e\rceil}\right)^{[1/p^e]}$ for all $e$ large enough.

Test ideals have many similar properties to multiplier ideals.
If $\mathfrak{a}\subseteq\mathfrak{b}$, then
$\tau(\mathfrak{a}^\lambda)\subseteq\tau(\mathfrak{b}^\lambda)$ for
all $\lambda\geqslant 0$. If $m$ is a positive integer, then
$\tau(\mathfrak{a}^{m\lambda})=\tau((\mathfrak{a}^m)^\lambda)$.

One can also define an asymptotic version of test ideals similar to
asymptotic multiplier ideals. Suppose that $\mathfrak{a}_\bullet$ is
a graded sequence of ideals on $X$
($\mathfrak{a}_m\cdot\mathfrak{a}_n\subseteq\mathfrak{a}_{m+n}$) and
$\lambda$ is a positive real number. If $m$ and $l$ are two positive
integers such that $\mathfrak{a}_m$ is nonzero, then
$$\tau(\mathfrak{a}_m^{\lambda/m})=\tau((\mathfrak{a}_m^l)^{\lambda/ml})
\subseteq\tau(\mathfrak{a}_{ml}^{\lambda/ml}).$$ By the Noetherian
property, there is a unique ideal
$\tau(\mathfrak{a}_\bullet^\lambda)$, called the \textbf{asymptotic
test ideal} of $\mathfrak{a}_\bullet$ of exponent $\lambda$, such
that
$\tau(\mathfrak{a}_\bullet^\lambda)=\tau(\mathfrak{a}_m^{\lambda/m})$
for all $m$ large enough and sufficiently divisible. 


For linear series, let $D$ be a Cartier divisor on $X$ such that
$h^0(X,\mathcal{O}_X(mD))\neq 0$ for some positive integer $m$. We
then define $\tau(\lambda\cdot
\|D\|)=\tau(\mathfrak{a}_\bullet^\lambda)$ where $\mathfrak{a}_m$ is
the base ideal of the linear series $|mD|$. Then by definition,
$\tau(\lambda/r\cdot\|rD\|)=\tau(\lambda\cdot\|D\|)$ for every
positive integer $r$. If $D$ is a $\mathbb{Q}$-divisor such that
$h^0(X,\mathcal{O}_X(mD))\neq 0$ for some positive integer $m$
satisfying that $mD$ is Cartier, then we put
$\tau(\lambda\cdot\|D\|)=\tau(\lambda/r\cdot\|rD\|)$ for some $r>0$
such that $rD$ is Cartier. 


\section{Fourier-Mukai Transform}

We recall some facts about the Fourier-Mukai Transform from
\cite{M81}. Let $A$ be an abelian variety of dimension $g$,
$\hat{A}$ be the dual abelian variety and $P$ be the normalized
Poincar\'{e} line bundle on $A\times\hat{A}$. For any
$\hat{a}\in\hat{A}$, we let $P_{\hat{a}}=P|_{A\times \hat{a}}$. The
\textbf{Fourier-Mukai functor} $R\hat{S}:D(A)\rightarrow D(\hat{A})$
is given by
$R\hat{S}(\mathcal{F})=Rp_{\hat{A},*}(p_A^*\mathcal{F}\otimes P)$.
There is a corresponding functor $RS:D(\hat{A})\rightarrow D(A)$
such that $RS\circ R\hat{S}=(-1_A)^*[-g]$ and $R\hat{S}\circ
RS=(-1_{\hat{A}})^*[-g]$.

We will need the following result.

\begin{proposition}\label{hacon}
Let $\mathcal{F}$ be a non-zero coherent sheaf on $A$ such that
$H^i(A, \mathcal{F}\otimes P_{\hat{a}})=0$ for all
$\hat{a}\in\hat{A}$ and all $i>0$. If $\mathcal{F}\rightarrow k(a)$ is a
surjective morphism for some $a\in A$, then the induced map $H^0(A,
\mathcal{F}\otimes P_{\hat{a}})\rightarrow k(a)$ is surjective for
general $\hat{a}\in\hat{A}$.
\end{proposition}

\begin{proof}
See \cite[2.1]{H}.
\end{proof}

\begin{remark}
By $k(a)$, we mean the trivial skyscraper sheaf supported on $a$. If
$F:A\rightarrow A$ is the absolute Frobenius map on $A$, then
$F_*k(a)$ is a skyscraper sheaf with fiber isomorphic to $k$
supported on $a$. But this $k$ is not the trivial vector space over
$k$. Hence, $F_*k(a)\neq k(a)$.
\end{remark}

\section{Vanishing Theorems}

Suppose $f:X\rightarrow A$ is a nontrivial morphism where $X$ is a
smooth variety of general type over an algebraic closed field $k$ of
characteristic $p>0$ and $A$ is an abelian variety. Let $K_X$ be a
canonical divisor. Since $K_X$ is big, we have $K_X\sim_\mathbb{Q}
H+E$ where $H$ is an ample $\mathbb{Q}$-divisor and $E$ is an
effective $\mathbb{Q}$-divisor. Let $\Delta=(1-\epsilon)K_X+\epsilon
E$, where $\epsilon\in\mathbb{Q}$ and $0<\epsilon<1$. Fix a positive
integer $l$ such that $l\Delta$ is Cartier. Although $\Delta$ is not
necessarily effective, since $K_X$ is big and $E$ is effective, we
have the Iitaka dimension $\kappa(X,l\Delta)\geqslant 0$. For any
positive integer $r$, let
$\mathcal{F}_r=\mathcal{O}_X((r+1)K_X)\otimes\tau(\|r\Delta\|)$.



Let $\mathfrak{a_m}$ be the base ideal of the linear series
$|ml\Delta|$. By the definition of the asymptotic test ideal, we can
fix a positive integer $m'$ sufficiently large and divisible such
that
$\tau(\|r\Delta\|)=\tau(\mathfrak{a}_\bullet^{r/l})=\tau(\mathfrak{a}_{m'}^{r/m'l})$.
We may assume $m'=rm$ for some positive integer $m$. Then
$\tau(\|r\Delta\|)=\tau(\mathfrak{a}_{rm}^{1/ml})$. For every $e\gg
0$, we have
$\tau(\mathfrak{a}_{rm}^{1/ml})=\left(\mathfrak{a}_{rm}^{\lceil
p^e/ml\rceil}\right)^{[1/p^e]}$. Hence, the iterated trace map
$\text{Tr}^e$ gives a surjection
$$\text{Tr}^e:F_*^e(\mathfrak{a}_{rm}^{\lceil
p^e/ml\rceil}\cdot \mathcal{O}_X(K_X))
\rightarrow\tau(\|r\Delta\|)\cdot \mathcal{O}_X(K_X).$$ Tensoring
with $\mathcal{O}_X(rK_X)$, we have a surjection
$$F_*^e(\mathfrak{a}_{rm}^{\lceil p^e/ml\rceil}\cdot
\mathcal{O}_X((rp^e+1)K_X)) \rightarrow\mathcal{F}_r.$$ Let
$\widetilde{\mathcal{F}_{r,e}}=\mathfrak{a}_{rm}^{\lceil
p^e/ml\rceil}\cdot \mathcal{O}_X((rp^e+1)K_X)$. Then the surjection
above is $F^e_*\widetilde{\mathcal{F}_{r,e}}\rightarrow
\mathcal{F}_r$. Since $\mathfrak{a}_{rm}$ is the base ideal of
$|rml\Delta|$, the evaluation gives a surjection
$$H^0(X,\mathcal{O}_X(rml\Delta))\otimes\mathcal{O}_X(-rml\Delta)
\rightarrow\mathfrak{a}_{rm},$$ hence a surjection
$$V_{r,e}\otimes\mathcal{O}_X(-rml\lceil
p^e/ml\rceil\Delta) \rightarrow\mathfrak{a}_{rm}^{\lceil
p^e/ml\rceil},$$ where $V_{r,e}=\text{Sym}^{\lceil
p^e/ml\rceil}H^0(X,\mathcal{O}_X(rml\Delta))$. Tensoring with
$\mathcal{O}_X((rp^e+1)K_X)$, we have a surjection
$$\mathcal{F}_{r,e}=V_{r,e}\otimes\mathcal{O}_X(-rml\lceil
p^e/ml\rceil\Delta+(rp^e+1)K_X) \rightarrow
\widetilde{\mathcal{F}_{r,e}},$$ hence a surjection
$F^e_*\mathcal{F}_{r,e}\rightarrow\mathcal{F}_r$ since $F^e$ is
affine.

\begin{lemma}\label{d-image} Fix $r>0$. Then
$R^if_*(F^e_*\mathcal{F}_{r,e})=0$ for all $i>0$ and all $e$ large
enough.
\end{lemma}

\begin{proof}
First, we prove that $R^if_*\mathcal{F}_{r,e}=0$ for all $i>0$ and
all $e$ large enough. Since $V_{r,e}$ is a vector space over $k$, we
only need to show that
$$R^if_*\mathcal{O}_X(-rml\lceil p^e/ml\rceil\Delta+(rp^e+1)K_X)=0.$$
But
$$\begin{array}{cl}
        & -rml\lceil p^e/ml\rceil\Delta+(rp^e+1)K_X \\
=\!\!\! & -rmls\Delta+(rmls-rt+1)K_X \\
=\!\!\! & (1-rt)K_X+rmls(K_X-\Delta),
\end{array}$$
where $s=\lceil p^e/ml\rceil$ and $0\leqslant t=mls-p^e<ml$.
Noticing that $K_X-\Delta\sim_\mathbb{Q}\epsilon H$ which is ample,
we may apply Serre vanishing. For each value of $t\in [0,ml-1]$, we
have $R^if_*\mathcal{O}_X((1-rt)K_X+rmls(K_X-\Delta))=0$ for all $s$
large enough, i.e., all $e$ large enough. Thus
$R^if_*\mathcal{F}_{r,e}=0$.

Now, since $F^e$ is exact and commutes with $f$, we have
$$R^if_*(F^e_*\mathcal{F}_{r,e}) = R^i(f\circ
F^e)_*\mathcal{F}_{r,e} = R^i(F^e \circ f)_*\mathcal{F}_{r,e} =
F^e_*(R^if_*\mathcal{F}_{r,e}) = 0$$ for all $i>0$ and all $e$ large
enough.
\end{proof}

\begin{lemma}\label{IT0} Fix $r>0$. There is an integer $M>0$ such that
$H^i(A, f_*(F^e_*\mathcal{F}_{r,e})\otimes P)=0$ for all $i>0$,
$e>M$ and $P\in Pic^0(A)$.
\end{lemma}

\begin{proof}
The proof is similar to that of Lemma \ref{d-image}. By Lemma
\ref{d-image} and the projection formula,
$R^if_*(F^e_*\mathcal{F}_{r,e}\otimes f^*P)=0$ for all $i>0$ and $e$
large enough. Hence, by a spectral sequence argument, it suffices to
prove that $H^i(X,F^e_*\mathcal{F}_{r,e}\otimes f^*P)=0$ or
equivalently that $H^i(X, \mathcal{F}_{r,e}\otimes f^*P^{\otimes
p^e})=0$. We only need to show that
$$H^i(X,\mathcal{O}_X(-rml\lceil
p^e/ml\rceil\Delta+(rp^e+1)K_X)\otimes f^*P^{\otimes p^e})=0.$$
Assume that $s=\lceil p^e/ml\rceil$ and $0\leqslant t=mls-p^e<ml$.
Since $K_X-\Delta\sim_\mathbb{Q}\epsilon H$ is ample, by Fujita
vanishing, for each value of $t$, there is an $M_t>0$ such that for
all $e>M_t$ and all nef line bundles $\mathcal{N}$ on $X$, we have
$$H^i(X,\mathcal{O}_X((1-rt)K_X+rmls(K_X-\Delta))\otimes\mathcal{N})=0.$$
Let $M=\max\{M_t\}$, then
$$H^i(X,\mathcal{O}_X(-rml\lceil
p^e/ml\rceil\Delta+(rp^e+1)K_X)\otimes\mathcal{N})=0$$ for all $e>M$
and all nef line bundles $\mathcal{N}$. In particular, we can take
$\mathcal{N}=f^*P^{\otimes p^e}$. The lemma follows.
\end{proof}

\section{Main Results}

Fix a positive integer $m$ such that
$\tau(\|r\Delta\|)=\tau(\mathfrak{a}_{rm}^{1/ml})$. Let
$\mathcal{I}$ be an ideal sheaf in $\mathcal{O}_X$. In our
applications, $\mathcal{I}=\mathcal{O}_X$ or
$\mathcal{I}=\mathcal{I}_x$, where $\mathcal{I}_x$ is the maximal
ideal of closed point. The composition of the two surjections,
$F^e_*\mathcal{F}_{r,e}\rightarrow \mathcal{F}_r$ and
$\mathcal{F}_r\rightarrow \mathcal{F}_r\otimes
\mathcal{O}_X/\mathcal{I}$, is still surjective. We define
$(F^e_*\mathcal{F}_{r,e})_{\mathcal{I}}$ to be the kernel of this
composition. Then
$(F^e_*\mathcal{F}_{r,e})_{\mathcal{O}_X}=F^e_*\mathcal{F}_{r,e}$.
Assuming that $$\text{the intersection of the co-supports of }
\tau(\|r\Delta\|) \text{ and } \mathcal{I} \text{ is
empty},\eqno{(*)_r}$$ since the composition
$(F^e_*\mathcal{F}_{r,e})_{\mathcal{I}}\rightarrow
F^e_*\mathcal{F}_{r,e}\rightarrow \mathcal{F}_r\rightarrow
\mathcal{F}_r\otimes \mathcal{O}_X/\mathcal{I}$ is 0, it factors
through the kernel of $\mathcal{F}_r\rightarrow \mathcal{F}_r\otimes
\mathcal{O}_X/\mathcal{I}$, which is $\mathcal{F}_r\otimes
\mathcal{I}$.
 We have a map
$(F^e_*\mathcal{F}_{r,e})_{\mathcal{I}}\rightarrow
\mathcal{F}_r\otimes \mathcal{I}$, and by the 5-lemma, it is
surjective. This is summarized in the following commutative diagram.
$$\xymatrix{
0 \ar[r] & (F^e_*\mathcal{F}_{r,e})_{\mathcal{I}} \ar[r]\ar@{->>}[d]
& F^e_*\mathcal{F}_{r,e} \ar[r]\ar@{->>}[d] & \mathcal{F}_r\otimes
\mathcal{O}_X/\mathcal{I}
\ar[r]\ar@{=}[d] & 0 \\
0 \ar[r] & \mathcal{F}_r\otimes \mathcal{I} \ar[r] & \mathcal{F}_r
\ar[r] & \mathcal{F}_r\otimes \mathcal{O}_X/\mathcal{I} \ar[r] & 0
}$$

\begin{remark}\label{rmk}
The condition $(*)_r$ is true if $\mathcal{I}=\mathcal{O}_X$ or
$\mathcal{I}=\mathcal{I}_x$ where $x$ is not in the co-support of
$\tau(\|r\Delta\|)$. And $(*)_r$ implies $(*)_s$ if $r\geqslant s$,
since $\tau(\|r\Delta\|)\subseteq \tau(\|s\Delta\|)$.
\end{remark}

Suppose that $x$ is a point in $X$ such that $x$ is not in the
co-support of the ideals $\tau(\|r\Delta\|)$ or $\mathcal{I}$. Then
the restriction to the point $x$ gives a surjection
$\mathcal{F}_r\otimes \mathcal{I}\rightarrow \mathcal{F}_r\otimes
k(x)\cong k(x)$. Hence, a surjection
$(F^e_*\mathcal{F}_{r,e})_{\mathcal{I}}\rightarrow
\mathcal{F}_r\otimes k(x)$. Let
$(F^e_*\mathcal{F}_{r,e})_{\mathcal{I},x}$ be the kernel. Thus, we
have the following exact sequence
$$0\rightarrow (F^e_*\mathcal{F}_{r,e})_{\mathcal{I},x}\rightarrow
(F^e_*\mathcal{F}_{r,e})_{\mathcal{I}}\rightarrow
\mathcal{F}_r\otimes k(x)\rightarrow 0.$$

Let $f:X\rightarrow A$ be a nontrivial separable morphism where $A$
is an abelian variety.

\begin{theorem}\label{bpf} Fix $e>0$ and $r$ a positive integer. Let $\mathcal{I}$ be an ideal sheaf in $\mathcal{O}_X$ satisfying $(*)_r$.
Suppose that $x$ is a point in $X$ such that
    \begin{enumerate}
    \item $x$ is not in the co-support of $\tau(\|r\Delta\|)$ or $\mathcal{I}$,
    \item $f_*(F^e_*\mathcal{F}_{r,e})_{\mathcal{I},x}\neq
    f_*(F^e_*\mathcal{F}_{r,e})_{\mathcal{I}}$,
    \item $H^i(A,f_*(F^e_*\mathcal{F}_{r,e})_{\mathcal{I}}\otimes P)=0$ for
    all $i>0$ and all $P\in \text{Pic}^0(A)$.
    \end{enumerate}
Then the homomorphism
$H^0(X,(F^e_*\mathcal{F}_{r,e})_{\mathcal{I}}\otimes
f^*P)\rightarrow H^0(X,\mathcal{F}_r\otimes \mathcal{I}\otimes
f^*P\otimes k(x))\cong k(x)$ induced by $\phi_{r,e,\mathcal{I}}$ is
surjective for general $P\in\text{Pic}^0(A)$. Moreover, $x$ is not a
base point of $\mathcal{F}_r\otimes\mathcal{I}\otimes f^*P$ for
general $P\in \text{Pic}^0(A)$.
\end{theorem}

\begin{proof}
Pushing forward the exact sequence $$0\rightarrow
(F^e_*\mathcal{F}_{r,e})_{\mathcal{I},x} \rightarrow
(F^e_*\mathcal{F}_{r,e})_{\mathcal{I}} \rightarrow
\mathcal{F}_r\otimes k(x) \rightarrow 0,$$ we have
$$0 \rightarrow f_*(F^e_*\mathcal{F}_{r,e})_{\mathcal{I},x} \rightarrow
f_*(F^e_*\mathcal{F}_{r,e})_{\mathcal{I}} \rightarrow
f_*(\mathcal{F}_r\otimes k(x)) \rightarrow
R^1f_*(F^e_*\mathcal{F}_{r,e})_{\mathcal{I},x}\rightarrow \cdots.$$
Let $a=f(x)$. Since $f$ is separable, we have that $a$ is reduced,
hence $f_*(\mathcal{F}_r\otimes k(x))\cong k(a)$. By assumption,
$f_*(F^e_*\mathcal{F}_{r,e})_{\mathcal{I},x} \rightarrow
f_*(F^e_*\mathcal{F}_{r,e})_{\mathcal{I}}$ is not an isomorphism,
which implies that the kernel of $k(a) \rightarrow
R^1f_*(F^e_*\mathcal{F}_{r,e})_{\mathcal{I},x}$ is not 0. But the
kernel is a sub-sheaf of $k(a)$ who has no non-zero sub-sheaf other
than itself. Hence the kernel is $k(a)$ and we have an exact
sequence
$$0 \rightarrow f_*(F^e_*\mathcal{F}_{r,e})_{\mathcal{I},x} \rightarrow
f_*(F^e_*\mathcal{F}_{r,e})_{\mathcal{I}} \rightarrow k(a)
\rightarrow 0.$$ Applying Proposition \ref{hacon} to the surjection
$f_*(F^e_*\mathcal{F}_{r,e})_{\mathcal{I}} \rightarrow k(a)$, we
have the surjection
$H^0(f_*(F^e_*\mathcal{F}_{r,e})_{\mathcal{I}}\otimes P)\rightarrow
k(a)$ for general $P\in \text{Pic}^0(A)$. Hence, the theorem
follows. For the moreover part, noticing that the surjection factors
through $H^0(\mathcal{F}_r\otimes\mathcal{I}\otimes f^*P)$, we have
the induced homomorphism $H^0(\mathcal{F}_r\otimes\mathcal{I}\otimes
f^*P)\rightarrow k(x)$ is also surjective.
\end{proof}

The following corollary is useful in the case of maximal Albanese dimension.

\begin{corollary}\label{gfinite}
Suppose $f$ is finite over an open subset $U$ in $A$. Fix $e>0$ and
$r$ a positive integer. Let $\mathcal{I}$ be an ideal sheaf in
$\mathcal{O}_X$ satisfying $(*)_r$. Suppose that $x$ is a point in
$X$ such that
    \begin{enumerate}
    \item $x$ is not in the co-support of $\tau(\|r\Delta\|)$ or $\mathcal{I}$,
    \item $a=f(x)\in U$,
    \item $H^i(A,f_*(F^e_*\mathcal{F}_{r,e})_{\mathcal{I}}\otimes P)=0$ for
    all $i>0$ and all $P\in \text{Pic}^0(A)$.
    \end{enumerate}
Then the conclusion of Theorem \ref{bpf} still holds.
\end{corollary}

\begin{proof}
By Theorem \ref{bpf}, we only need to show that
$f_*(F^e_*\mathcal{F}_{r,e})_{\mathcal{I},x}\neq
f_*(F^e_*\mathcal{F}_{r,e})_{\mathcal{I}}$. Recall that we have an
exact sequence
$$0 \rightarrow f_*(F^e_*\mathcal{F}_{r,e})_{\mathcal{I},x} \rightarrow
f_*(F^e_*\mathcal{F}_{r,e})_{\mathcal{I}} \rightarrow
f_*(\mathcal{F}_r\otimes k(x)) \rightarrow
R^1f_*(F^e_*\mathcal{F}_{r,e})_{\mathcal{I},x}\rightarrow \cdots.$$
If
$f_*(F^e_*\mathcal{F}_{r,e})_{\mathcal{I},x}=f_*(F^e_*\mathcal{F}_{r,e})_{\mathcal{I}}$,
we have that the map $f_*(\mathcal{F}_r\otimes k(x)) \rightarrow
R^1f_*(F^e_*\mathcal{F}_{r,e})_{\mathcal{I},x}$ is nonzero. So the
stalk of $R^1f_*(F^e_*\mathcal{F}_{r,e})_{\mathcal{I},x}$ at $a$ is
nonzero. But as $f$ is finite over $a$, the higher direct images are
0 at $a$, a contradiction.
\end{proof}

\begin{remark}\label{var_gfinite}
It is easy to see that the proof of Theorem \ref{bpf} and Corollary
\ref{gfinite} not only works for the surjection
$(F^e_*\mathcal{F}_{r,e})_\mathcal{I}\rightarrow
\mathcal{F}_r\otimes k(x)$, but any surjection to the trivial
skyscraper sheaf satisfying the vanishing condition (3). We will use
this variant version repeatedly in the proof of Theorem \ref{main}.
\end{remark}

Theorem \ref{bpf} also gives information on the base locus of
$\mathcal{O}_X(2K_X)\otimes f^*P$ for general $P\in\text{Pic}^0(A)$.

\begin{corollary}\label{cor}
Fix $e>M$ as in Lemma \ref{IT0}. Suppose that $x$ is a
point in $X$ such that
    \begin{enumerate}
    \item $x$ is not in the co-support of $\tau(\|\Delta\|)$,
    \item $f_*(F^e_*\mathcal{F}_{1,e})_{\mathcal{O}_X,x}\neq
    f_*(F^e_*\mathcal{F}_{1,e})$.
    \end{enumerate}
Then $x$ is not a base point of $\mathcal{F}_1\otimes f^*P$ for
general $P\in \text{Pic}^0(A)$. Hence $x$ is not a base point of
$\mathcal{O}_X(2K_X)\otimes f^*P$ for general $P\in
\text{Pic}^0(A)$.
\end{corollary}

\begin{proof}
The first part of the corollary follows directly from Theorem
\ref{bpf} and Lemma \ref{IT0}. The second part follows from the
facts that $\mathcal{F}_1=\mathcal{O}_X(2K_X)\otimes
\tau(\|\Delta\|)$ and $\tau(\|\Delta\|)$ is an ideal.
\end{proof}

We are ready to prove the main result.

\begin{theorem}\label{main}
Let $X$ be a smooth projective variety of general type over an
algebraic closed field $k$ of characteristic $p>0$. If $X$ has
maximal Albanese dimension and the Albanese map is separable, then
$|4K_X|$ induces a birational map.
\end{theorem}

\begin{proof}
Suppose $A$ is the Albanese variety and $f:X\rightarrow A$ is the
Albanese map. Since $f$ is generically finite, there is an open
subset $U$ of $A$ such that $f$ is finite over $U$. And since $f$ is
separable, we can fix a canonical divisor
$K_X=f^*K_A+R_f=R_f\geqslant 0$. As usual, we let
$\mathcal{F}_r=\mathcal{O}_X((r+1)K_X)\otimes\tau(\|r\Delta\|)$ and
$\mathfrak{a}_m=\mathfrak{bs}(|ml\Delta|)$. We fix a positive
integer $m$ such that
$\tau(\|r\Delta\|)=\tau(\mathfrak{a}_{rm}^{1/ml})$ for $r=1,2,3$ and
define
$$\widetilde{\mathcal{F}_{r,e}}=\mathfrak{a}_{rm}^{\lceil p^e/ml\rceil}\cdot
\mathcal{O}_X((rp^e+1)K_X),$$
$$\mathcal{F}_{r,e}=V_{r,e}\otimes\mathcal{O}_X(-rml\lceil
p^e/ml\rceil\Delta+(rp^e+1)K_X),$$ where $V_{r,e}=\text{Sym}^{\lceil
p^e/ml\rceil}H^0(X,\mathcal{O}_X(rml\Delta))$. Let
$$\widetilde{\mathcal{F}_{1,e}}^-=\mathfrak{a}_m^{\lceil p^e/ml\rceil}\cdot
\mathcal{O}_X(p^eK_X)=\widetilde{\mathcal{F}_{1,e}}\otimes
\mathcal{O}_X(-K_X),$$
$$\mathcal{F}_{1,e}^-=V_{1,e}\otimes\mathcal{O}_X(-ml\lceil
p^e/ml\rceil\Delta+p^eK_X)=\mathcal{F}_{1,e}\otimes
\mathcal{O}_X(-K_X).$$

\vspace{0.3cm}\textbf{Claim:} For $e\gg 0$, we have
$R^if_*(F^e_*\mathcal{F}_{1,e}^-)=0$ and
$H^i(A,f_*(F^e_*\mathcal{F}_{1,e}^-)\otimes P)=0$ for all $i>0$ and
$P\in\text{Pic}^0(A)$.

The claim follows immediately from the proofs of Lemma \ref{d-image}
and \ref{IT0}.

\vspace{0.3cm}We fix an integer $e\gg 0$ such that
$H^i(A,f_*\mathcal{G}\otimes P)=0$ holds for all $i>0$ and all
$\mathcal{G}\in\{\mathcal{F}_{1,e}, \mathcal{F}_{2,e},
\mathcal{F}_{3,e}, \mathcal{F}_{1,e}^-\}$. By a general point $a\in
A$, we mean a point $a\in U$. By a general point $x\in X$, we mean a
point $x\in f^{-1}(U)$ such that $x$ is not in the co-supports of
$\mathfrak{a}_m^{\lceil p^e/ml\rceil}$ or $\tau(\|3\Delta\|)$
(hence, not in the co-supports of $\tau(\|2\Delta\|)$ or
$\tau(\|\Delta\|)$ by Remark \ref{rmk}) and that $x$ is not in the
support of $K_X=R_f$. (It is not hard to see that $x$ is not in the
co-support of $\mathfrak{a}_m$ is equivalent to that $x$ is not in
the co-support of $\mathfrak{a}_m^{\lceil p^e/ml\rceil}$ and implies
that $x$ is not in the co-support of $\tau(\|3\Delta\|)$.)

Our strategy is: First, by Theorem \ref{bpf}, we have that $x$ is
not a base point of $\mathcal{F}_1\otimes f^*P$ for general
$P\in\text{Pic}^0(A)$. Then, by comparing $\mathcal{F}_1$ and
$\mathcal{F}_2$ via $\mathcal{F}_{1,e}^-$, we show that $x$ is not a
base point of $\mathcal{F}_2\otimes f^*P$ for all
$P\in\text{Pic}^0(A)$. Using this fact, we show that
$\mathcal{F}_2\otimes f^*P$ separates points for general
$P\in\text{Pic}^0(A)$. Finally, by comparing $\mathcal{F}_2$ and
$\mathcal{F}_3$ via $\mathcal{F}_{1,e}^-$, we have that
$\mathcal{F}_3\otimes f^*P$ separates points for all
$P\in\text{Pic}^0(A)$. Hence, so does $\mathcal{F}_3$. Following the
same idea, we show that $\mathcal{F}_3$ separates tangent vectors.

\vspace{0.3cm}\emph{Step 1.} By Lemma \ref{IT0} and Corollary
\ref{gfinite} with $r=1$ and $\mathcal{I}=\mathcal{O}_X$, we have
that for general $x\in X$, the homomorphism
$$H^0(X,F^e_*\mathcal{F}_{1,e}\otimes f^*P)\rightarrow
H^0(X,\mathcal{F}_1\otimes f^*P\otimes k(x))\cong k(x)$$ is
surjective for general $P\in \text{Pic}^0(A)$.

\vspace{0.3cm}\emph{Step 2.} We show that for general $x\in X$, the
homomorphism
$$H^0(X,F^e_*\mathcal{F}_{2,e}\otimes f^*Q)\rightarrow
H^0(X,\mathcal{F}_2\otimes f^*Q\otimes k(x))\cong k(x)$$ is
surjective for all $Q\in\text{Pic}^0(A)$.

Let us give a quick explanation of the idea in this step first. We
can pick $P\in \text{Pic}^0(A)$ such that both $P$ and $Q\otimes
P^\vee$ are general. We have already shown that there is a global
section of $F^e_*\mathcal{F}_{1,e}\otimes f^*P$ which induces a
global section of $\mathcal{F}_1\otimes f^*P$ not vanishing at $x$.
Notice that the difference between $\mathcal{F}_1\otimes f^*P$ and
$\mathcal{F}_2\otimes f^*Q$ near a general point $x$ is
$\mathcal{O}_X(K_X)\otimes f^*(Q\otimes P^\vee)$. If we can find a
global section of $\mathcal{O}_X(K_X)$ not vanishing at $x$, we can
obtain a global section of $\mathcal{F}_2\otimes f^*Q$ not vanishing
at $x$. This can be done as $K_X$ is effective. But, unfortunately,
$\mathcal{O}_X(K_X)$ does not behave well globally with the
Frobenius, $\widetilde{\mathcal{F}_{r,e}}$ and $\mathcal{F}_{r,e}$.
We have to introduce $\widetilde{\mathcal{F}_{1,e}}^-$ and
$\mathcal{F}_{1,e}^-$ as the bridge from
$\widetilde{\mathcal{F}_{1,e}}$ to $\widetilde{\mathcal{F}_{2,e}}$
and $\mathcal{F}_{1,e}$ to $\mathcal{F}_{2,e}$, respectively. The
induced map $F^e_*\mathcal{F}_{1,e}\rightarrow
F^e_*\mathcal{F}_{2,e}$ is commutative with
$\mathcal{F}_1\rightarrow \mathcal{O}_X(K_X)\otimes
\mathcal{F}_1\cong \mathcal{F}_2$ near a general point $x$ by the
projection formula. Hence, we view $F^e_*\mathcal{F}_{1,e}^-$ as
giving a homomorphism $\mathcal{F}_1\otimes k(x)\rightarrow
\mathcal{F}_2\otimes k(x)$. We only need to show that there is a
global section of $F^e_*\mathcal{F}_{1,e}^-$ inducing a non-zero
homomorphism $\mathcal{F}_1\otimes k(x)\rightarrow
\mathcal{F}_2\otimes k(x)$. Here is the detailed proof.

Since $\mathfrak{a}_m^{\lceil p^e/ml\rceil}$ is an ideal, we have an
inclusion $\widetilde{\mathcal{F}_{1,e}}^-\rightarrow
\mathcal{O}_X(p^eK_X)$. Tensoring with the vector bundle
$\mathcal{F}_{1,e}$, we have an inclusion
$$\widetilde{\mathcal{F}_{1,e}}^-\otimes\mathcal{F}_{1,e}\rightarrow
\mathcal{O}_X(p^eK_X)\otimes\mathcal{F}_{1,e},$$ whose cokernel is
supported on the co-support of $\mathfrak{a}_m^{\lceil
p^e/ml\rceil}$. Pushing forward by the Frobenius, we get another
inclusion
$$F^e_*(\widetilde{\mathcal{F}_{1,e}}^-\otimes\mathcal{F}_{1,e})\rightarrow
F^e_*(\mathcal{O}_X(p^eK_X)\otimes\mathcal{F}_{1,e})\cong
\mathcal{O}_X(K_X)\otimes F^e_*(\mathcal{F}_{1,e}),$$ whose cokernel
is still supported on the co-support of $\mathfrak{a}_m^{\lceil
p^e/ml\rceil}$. Hence, the induced morphism
$$\alpha:F^e_*(\widetilde{\mathcal{F}_{1,e}}^-\otimes\mathcal{F}_{1,e})\otimes k(x)\rightarrow
\mathcal{O}_X(K_X)\otimes F^e_*(\mathcal{F}_{1,e})\otimes k(x)$$ is
an isomorphism providing that $x$ is general. Since
$\mathfrak{a}_m^2\subseteq \mathfrak{a}_{2m}$, we have a morphism
$\widetilde{\mathcal{F}_{1,e}}^-\otimes\widetilde{\mathcal{F}_{1,e}}\rightarrow
\widetilde{\mathcal{F}_{2,e}}$. Combining with the morphism
$\mathcal{F}_{1,e}\rightarrow \widetilde{\mathcal{F}_{1,e}}$, we
have that
$$\widetilde{\mathcal{F}_{1,e}}^-\otimes\mathcal{F}_{1,e}\rightarrow
\widetilde{\mathcal{F}_{1,e}}^-\otimes\widetilde{\mathcal{F}_{1,e}}\rightarrow
\widetilde{\mathcal{F}_{2,e}}.$$

On the other hand, since $\tau(\|\Delta\|)$ and $\tau(\|2\Delta\|)$
are ideals, the induced inclusions $\mathcal{F}_1\otimes
k(x)\rightarrow \mathcal{O}_X(2K_X)\otimes k(x)$ and
$\mathcal{F}_2\otimes k(x)\rightarrow \mathcal{O}_X(3K_X)\otimes
k(x)$ are both isomorphisms providing that $x$ is general. Hence,
there is a morphism $$\mathcal{F}_1\otimes k(x)\rightarrow
\mathcal{O}_X(K_X)\otimes\mathcal{F}_1\otimes k(x) \cong
\mathcal{O}_X(3K_X)\otimes k(x)\cong \mathcal{F}_2\otimes k(x).$$

Combining the discussion above and the trace maps
$F^e_*\mathcal{F}_{r,e}\rightarrow
F^e_*\widetilde{\mathcal{F}_{r,e}}\rightarrow \mathcal{F}_r$, we
have the following commutative diagram: {\tiny $$\xymatrix{
F^e_*\mathcal{F}_{1,e}\otimes k(x) \ar[r]\ar@{->>}[d] &
\mathcal{O}_X(K_X)\otimes F^e_*\mathcal{F}_{1,e}\otimes k(x)
\ar[r]_-{\alpha^{-1}}^-\simeq\ar@{->>}[d] &
F^e_*(\widetilde{\mathcal{F}_{1,e}}^-\otimes\mathcal{F}_{1,e})\otimes
k(x) \ar[r]\ar@{->>}[d] & F^e_*\widetilde{\mathcal{F}_{2,e}}\otimes
k(x)
\ar@{->>}[d] \\
\mathcal{F}_1\otimes k(x) \ar[r]^-\simeq &
\mathcal{O}_X(K_X)\otimes\mathcal{F}_1\otimes k(x) \ar@{=}[r] &
\mathcal{O}_X(K_X)\otimes\mathcal{F}_1\otimes k(x) \ar[r]^-\simeq &
\mathcal{F}_2\otimes k(x). }$$} The surjectivities of the first,
second and last vertical maps are induced by the surjectivity of
$F^e_*\mathcal{F}_{r,e}\rightarrow \mathcal{F}_r$. The third map is
the same as the second map.

Noticing that $\mathcal{F}_{1,e}$ is a vector bundle, we have that
the morphism
$\widetilde{\mathcal{F}_{1,e}}^-\otimes\mathcal{F}_{1,e}\rightarrow
\widetilde{\mathcal{F}_{2,e}}$ is equivalent to a morphism
$\widetilde{\mathcal{F}_{1,e}}^-\rightarrow
\mathcal{H}om_{\mathcal{O}_X}(\mathcal{F}_{1,e},\widetilde{\mathcal{F}_{2,e}})$.
We have $$\begin{array}{r@{\,\rightarrow\,}l}
\vspace{0.2cm}F^e_*\widetilde{\mathcal{F}_{1,e}}^- &
F^e_*\mathcal{H}om_{\mathcal{O}_X}(\mathcal{F}_{1,e},\widetilde{\mathcal{F}_{2,e}})
\\
\vspace{0.2cm}& \mathcal{H}om_{\mathcal{O}_X}(F^e_*\mathcal{F}_{1,e},F^e_*\widetilde{\mathcal{F}_{2,e}}) \\
\vspace{0.2cm}&
\mathcal{H}om_{\mathcal{O}_X}(F^e_*\mathcal{F}_{1,e}\otimes
k(x),F^e_*\widetilde{\mathcal{F}_{2,e}}\otimes k(x))
\end{array}$$ which, by construction, is how $F^e_*\widetilde{\mathcal{F}_{1,e}}^-$ induces the top row of the
commutative diagram above. This induces a morphism between
$\mathcal{F}_1\otimes k(x)$ and $\mathcal{F}_2\otimes k(x)$. Indeed,
for any $a\in \mathcal{F}_1\otimes k(x)$, we have some $b\in
F^e_*\mathcal{F}_{1,e}\otimes k(x)$ (maybe not unique) mapped to $a$
by the first vertical arrow in the commutative diagram. Applying the
top row induced by $F^e_*\widetilde{\mathcal{F}_{1,e}}^-$ and then
the last vertical arrow on $b$, we get some $c\in
\mathcal{F}_2\otimes k(x)$ which is independent of the choice of $b$
since the diagram commutes. Hence, we have a morphism
$$F^e_*\widetilde{\mathcal{F}_{1,e}}^- \rightarrow
\mathcal{H}om_{\mathcal{O}_X}(\mathcal{F}_1\otimes
k(x),\mathcal{F}_2\otimes k(x))\cong k(x).$$

\begin{remark}
We should point out that, on any affine open set $V$ in $X$, the map
that we constructed above between $F^e_*\mathcal{F}_{1,e}$ and
$F^e_*\widetilde{\mathcal{F}_{2,e}}$ is not
$$F^e_*\mathcal{F}_{1,e}\rightarrow
F^e_*\widetilde{\mathcal{F}_{1,e}}^- \otimes F^e_*\mathcal{F}_{1,e}
\rightarrow F^e_*(\widetilde{\mathcal{F}_{1,e}}^- \otimes
\mathcal{F}_{1,e}) \rightarrow F^e_*\widetilde{\mathcal{F}_{2,e}},$$
where the second map is the natural morphism of pushing forward a
tensor product. The map that we constructed is
$$F^e_*\mathcal{F}_{1,e}\rightarrow F^e_*(\widetilde{\mathcal{F}_{1,e}}^- \otimes
\mathcal{F}_{1,e})\rightarrow F^e_*\widetilde{\mathcal{F}_{2,e}},$$
where the first map is by pushing forward
$\mathcal{F}_{1,e}\rightarrow \widetilde{\mathcal{F}_{1,e}}^-
\otimes \mathcal{F}_{1,e}$.
\end{remark}

Assuming that $x$ is not in the support of the effective divisor
$K_X$ that we fixed as the ramification divisor before, since the
bottom row of the commutative diagram are all isomorphisms in this
case, the morphism above $F^e_*\widetilde{\mathcal{F}_{1,e}}^-
\rightarrow k(x)$ is nonzero, and hence surjective. Combining with
the surjection $\mathcal{F}_{1,e}^-\rightarrow
\widetilde{\mathcal{F}_{1,e}}^-$, we have the following surjection:
$$F^e_*\mathcal{F}_{1,e}^-\rightarrow \mathcal{H}om_{\mathcal{O}_X}(\mathcal{F}_1\otimes
k(x),\mathcal{F}_2\otimes k(x))\cong k(x).$$

For any $Q\in \text{Pic}^0(A)$, we can pick $P\in \text{Pic}^0(A)$
such that $P$ and $Q\otimes P^\vee$ are both general. Applying
Corollary \ref{gfinite} and Remark \ref{var_gfinite} to the
surjection $F^e_*\mathcal{F}_{1,e}^-\rightarrow
\mathcal{H}om_{\mathcal{O}_X}(\mathcal{F}_1\otimes
k(x),\mathcal{F}_2\otimes k(x))$, since $Q\otimes P^\vee$ is
general, we get a surjection
$$H^0(X,F^e_*\mathcal{F}_{1,e}^-\otimes f^*(Q\otimes
P^\vee))\rightarrow \text{Hom}_{\mathcal{O}_X}(\mathcal{F}_1\otimes
f^*P\otimes k(x),\mathcal{F}_2\otimes f^*Q\otimes k(x)).$$ Combining
with the fact from Step 1, that $H^0(X,F^e_*\mathcal{F}_{1,e}\otimes
f^*P)\rightarrow H^0(X,\mathcal{F}_1\otimes f^*P\otimes k(x))$ is
surjective, we have a surjection
$$H^0(X,F^e_*\mathcal{F}_{1,e}^-\otimes F^e_*\mathcal{F}_{1,e}\otimes
f^*Q)\rightarrow H^0(X,\mathcal{F}_2\otimes f^*Q\otimes k(x)).$$
Notice that there is a natural commutative diagram $$\xymatrix{
\mathcal{F}_{1,e}^-\otimes \mathcal{F}_{1,e} \ar[r]\ar[d] &
\mathcal{F}_{2,e} \ar[dd]\\
\widetilde{\mathcal{F}_{1,e}}^-\otimes \mathcal{F}_{1,e} \ar[d] & \\
\widetilde{\mathcal{F}_{1,e}}^-\otimes \widetilde{\mathcal{F}_{1,e}}
\ar[r] & \widetilde{\mathcal{F}_{2,e}} }$$ By construction,
$F^e_*\mathcal{F}_{1,e}^-\otimes F^e_*\mathcal{F}_{1,e}\rightarrow
\mathcal{F}_2$ factors through $F^e_*\mathcal{F}_{2,e}$. Therefore,
we have that the homomorphism
$$H^0(X,F^e_*\mathcal{F}_{2,e}\otimes f^*Q)\rightarrow
H^0(X,\mathcal{F}_2\otimes f^*Q\otimes k(x))$$ is surjective for all
$Q\in \text{Pic}^0(A)$.

\vspace{0.3cm}\emph{Step 3.} We show that for general $x_1, x_2\in
X$ and general $P\in \text{Pic}^0(A)$, we can find a section in
$F^e_*\mathcal{F}_{2,e}\otimes f^*P$ which induces a section in
$\mathcal{F}_2\otimes f^*P$ vanishing at $x_1$ but not at $x_2$.

We only need to show that the map
$$H^0(X,(F^e_*\mathcal{F}_{2,e})_{\mathcal{I}_{x_1}}\otimes
f^*P)\rightarrow H^0(X,\mathcal{F}_2\otimes \mathcal{I}_{x_1}\otimes
f^*P\otimes k(x_2))\cong k(x_2)$$ is surjective. Noticing that $x_1$
is not in the co-support of $\tau(\|2\Delta\|)$, we have that
$\mathcal{I}_{x_1}$ satisfies $(*)_2$. Applying Corollary
\ref{gfinite} with $r=2$ and $\mathcal{I}=\mathcal{I}_{x_1}$, it
suffices to check
$H^i(A,f_*(F^e_*\mathcal{F}_{2,e})_{\mathcal{I}_{x_1}}\otimes P)=0$
for all $i>0$ and all $P\in\text{Pic}^0(A)$.

First we show that
$R^if_*(F^e_*\mathcal{F}_{2,e})_{\mathcal{I}_{x_1}}=0$ for $i>0$.
Pushing forward the exact sequence
$$0\rightarrow(F^e_*\mathcal{F}_{2,e})_{\mathcal{I}_{x_1}}\rightarrow F^e_*\mathcal{F}_{2,e}\rightarrow
\mathcal{F}_2\otimes k(x_1)\rightarrow 0$$ gives
$$f_*(F^e_*\mathcal{F}_{2,e})\rightarrow k(a_1)\rightarrow R^1f_*(F^e_*\mathcal{F}_{2,e})_{\mathcal{I}_{x_1}}\rightarrow
R^1f_*(F^e_*\mathcal{F}_{2,e})=0$$ and
$$R^if_*(F^e_*\mathcal{F}_{2,e})_{\mathcal{I}_{x_1}}\cong
R^if_*(F^e_*\mathcal{F}_{2,e})=0$$ for all $i\geqslant 2$, where
$a_1=f(x_1)$ and the vanishings follow from Lemma \ref{d-image}. As
in the proof of Theorem \ref{bpf} (notice that
$(F^e_*\mathcal{F}_{2,e})_{\mathcal{I}_{x_1}}=(F^e_*\mathcal{F}_{2,e})_{\mathcal{O}_X,x_1}$),
one sees that $f_*(F^e_*\mathcal{F}_{2,e})\rightarrow k(a_1)$ is
surjective, so
$R^1f_*(F^e_*\mathcal{F}_{2,e})_{\mathcal{I}_{x_1}}=0$.

Now, by a spectral sequence argument, we only need to show that
$H^i(X,(F^e_*\mathcal{F}_{2,e})_{\mathcal{I}_{x_1}}\otimes f^*P)=0$
for all $i>0$ and all $P\in\text{Pic}^0(A)$. We have the short exact
sequence
$$0\rightarrow(F^e_*\mathcal{F}_{2,e})_{\mathcal{I}_{x_1}}\otimes f^*P
\rightarrow F^e_*\mathcal{F}_{2,e}\otimes f^*P\rightarrow
\mathcal{F}_2\otimes f^*P \otimes k(x_1)\rightarrow 0.$$ By taking
the cohomology, we have
$$H^0(X,F^e_*\mathcal{F}_{2,e}\otimes f^*P)\rightarrow
k(x_1)\rightarrow
H^1(X,(F^e_*\mathcal{F}_{2,e})_{\mathcal{I}_{x_1}}\otimes
f^*P)\rightarrow H^1(X,F^e_*\mathcal{F}_{2,e}\otimes f^*P)=0$$ and
$$H^i(X,(F^e_*\mathcal{F}_{2,e})_{\mathcal{I}_{x_1}}\otimes f^*P)\cong
H^i(X,F^e_*\mathcal{F}_{2,e}\otimes f^*P)=0,$$ for all $i\geqslant
2$ where the vanishings follow from Lemma \ref{IT0}. Since by Step
2, $H^0(X,F^e_*\mathcal{F}_{2,e}\otimes f^*P)\rightarrow k(x_1)$ is
surjective, we have
$H^1(X,(F^e_*\mathcal{F}_{2,e})_{\mathcal{I}_{x_1}}\otimes f^*P)=0$.

\vspace{0.3cm}\emph{Step 4.} We show that for general $x_1, x_2\in
X$ and all $Q\in \text{Pic}^0(A)$, we can find a section in
$F^e_*\mathcal{F}_{3,e}\otimes f^*Q$ which induces a section in
$\mathcal{F}_3\otimes f^*Q$ vanishing at $x_1$ but not at $x_2$.

For any general points $x_1$ and $x_2$ and any
$Q\in\text{Pic}^0(A)$, we may pick $P\in\text{Pic}^0(A)$ such that
$P$ and $Q\otimes P^\vee$ are both general. Similar to Step 2, for
$i=1$ or 2, we have the following commutative diagram: {\tiny
$$\xymatrix{ F^e_*\mathcal{F}_{2,e}\otimes k(x_i) \ar[r]\ar[d] &
\mathcal{O}_X(K_X)\otimes F^e_*\mathcal{F}_{2,e}\otimes k(x_i)
\ar[r]^-\simeq\ar[d] &
F^e_*(\widetilde{\mathcal{F}_{1,e}}^-\otimes\mathcal{F}_{2,e})\otimes
k(x_i) \ar[r]\ar[d] & F^e_*\widetilde{\mathcal{F}_{3,e}}\otimes
k(x_i)
\ar[d] \\
\mathcal{F}_2\otimes k(x_i) \ar[r]^-\simeq &
\mathcal{O}_X(K_X)\otimes\mathcal{F}_2\otimes k(x_i) \ar@{=}[r] &
\mathcal{O}_X(K_X)\otimes\mathcal{F}_2\otimes k(x_i) \ar[r]^-\simeq
& \mathcal{F}_3\otimes k(x_i). }$$} We have that
$$F^e_*\mathcal{F}_{1,e}^- \rightarrow
\mathcal{H}om_{\mathcal{O}_X}(\mathcal{F}_2\otimes
k(x_i),\mathcal{F}_3\otimes k(x_i)) \cong k(x_i)$$ is surjective. We
may apply Corollary \ref{gfinite} and Remark \ref{var_gfinite} and
get that
$$H^0(X,F^e_*\mathcal{F}_{1,e}^-\otimes f^*(Q\otimes
P^\vee))\rightarrow \text{Hom}_{\mathcal{O}_X}(\mathcal{F}_2\otimes
f^*P\otimes k(x_2),\mathcal{F}_3\otimes f^*Q\otimes k(x_2))$$ is
surjective.

By Step 3, we have a section $s\in H^0(X,
F^e_*\mathcal{F}_{2,e}\otimes f^*P)$ restricting to 0 in
$\mathcal{F}_2\otimes f^*P\otimes k(x_1)$ and to nonzero in
$\mathcal{F}_2\otimes f^*P\otimes k(x_2)$. By the discussion above,
we have a section $s^-\in H^0(X,F^e_*\mathcal{F}_{1,e}^-\otimes
f^*(Q\otimes P^\vee))$ inducing a nonzero homomorphism between
$\mathcal{F}_2\otimes f^*P\otimes k(x_2)$ and $\mathcal{F}_3\otimes
f^*Q\otimes k(x_2)$. Hence, $s^-\otimes s$ gives a section in
$H^0(X, F^e_*\mathcal{F}_{3,e}\otimes f^*Q)$ restricting to 0 in
$\mathcal{F}_3\otimes f^*Q\otimes k(x_1)$ and to nonzero in
$\mathcal{F}_3\otimes f^*Q\otimes k(x_2)$.

\vspace{0.3cm}\emph{Step 5.} By Step 4, for all
$Q\in\text{Pic}^0(A)$, we have a surjection
$$H^0(X,F^e_*\mathcal{F}_{3,e}\otimes f^*Q)\rightarrow
H^0(X,\mathcal{F}_3\otimes f^*Q\otimes k(x_1,x_2)),$$ where
$k(x_1,x_2)$ is the skyscraper sheaf supported on $\{x_1,x_2\}$.
Since this surjection factors through $H^0(X,\mathcal{F}_3\otimes
f^*Q)$, we have that $\mathcal{F}_3\otimes f^*Q$ separates general
points for all $Q\in\text{Pic}^0(A)$.

\vspace{0.3cm}\emph{Step 6.} We show that for general $x\in X$, any
irreducible length two zero dimensional scheme $z$ with support $x$
and general $P\in\text{Pic}^0(A)$, we can find a section in
$(F^e_*\mathcal{F}_{2,e})_{\mathcal{I}_x}\otimes f^*P$ which induces
a section in $\mathcal{F}_2\otimes f^*P\otimes \mathcal{I}_x$ not
vanishing at $z$.

Let $r\in\{2,3\}$ and $\mathcal{I}_z$ be the ideal sheaf of $z$ in
$X$. Since $x$ is not in the co-support of $\tau(\|r\Delta\|)$, the
natural map $\mathcal{F}_r\otimes \mathcal{I}_x\rightarrow
\mathcal{F}_r\otimes \mathcal{I}_x/\mathcal{I}_z$ is surjective with
kernel $\mathcal{F}_r\otimes \mathcal{I}_z$. Recall that we have a
surjection $(F^e_*\mathcal{F}_{r,e})_{\mathcal{I}_x}\rightarrow
\mathcal{F}_r\otimes \mathcal{I}_x$. Hence the composition
$$(F^e_*\mathcal{F}_{r,e})_{\mathcal{I}_x}\rightarrow
\mathcal{F}_r\otimes \mathcal{I}_x\rightarrow \mathcal{F}_r\otimes
\mathcal{I}_x/\mathcal{I}_z$$ is surjective. We define
$(F^e_*\mathcal{F}_{r,e})_{\mathcal{I}_x,z}$ as the kernel of the
composition above. Since the composition
$(F^e_*\mathcal{F}_{r,e})_{\mathcal{I}_x,z}\rightarrow
(F^e_*\mathcal{F}_{r,e})_{\mathcal{I}_x}\rightarrow
\mathcal{F}_r\otimes\mathcal{I}_x\rightarrow
\mathcal{F}_r\otimes\mathcal{I}_x/\mathcal{I}_z$ is 0, it factors
through $\mathcal{F}_r\otimes \mathcal{I}_z$. By the 5-lemma, the
induced map $(F^e_*\mathcal{F}_{r,e})_{\mathcal{I}_x,z}\rightarrow
\mathcal{F}_r\otimes \mathcal{I}_z$ is surjective. This is
summarized in the following commutative diagram.

$$\xymatrix{
0 \ar[r] & (F^e_*\mathcal{F}_{r,e})_{\mathcal{I}_x,z}
\ar[r]\ar@{->>}[d] & (F^e_*\mathcal{F}_{r,e})_{\mathcal{I}_x}
\ar[r]\ar@{->>}[d] & \mathcal{F}_r\otimes
\mathcal{I}_x/\mathcal{I}_z
\ar[r]\ar@{=}[d] & 0 \\
0 \ar[r] & \mathcal{F}_r\otimes \mathcal{I}_z \ar[r] &
\mathcal{F}_r\otimes \mathcal{I}_x \ar[r] & \mathcal{F}_r\otimes
\mathcal{I}_x/\mathcal{I}_z \ar[r] & 0 }$$

To show the claim at the beginning of this step, we only need to
show that the map
$$H^0(X,(F^e_*\mathcal{F}_{2,e})_{\mathcal{I}_x}\otimes f^*P)\rightarrow
H^0(X,\mathcal{F}_2\otimes\mathcal{I}_x/\mathcal{I}_z\otimes
f^*P)\cong k(x)$$ is surjective. Suppose $a=f(x)$ and $t=f(z)$.
Since $f$ is separable and $x$ is not in the co-support of
$\tau(\|2\Delta\|)$, we have that
$f_*(\mathcal{F}_2\otimes\mathcal{I}_x/\mathcal{I}_z)\cong k(a)$
which is the trivial skyscraper sheaf at $a$. Noticing that
$\mathcal{I}_x$ satisfies $(*)_2$, it is not hard to see that the
proof of Theorem \ref{bpf} and Corollary \ref{gfinite} still works
for $r=2$, $\mathcal{I}=\mathcal{I}_x$ and the surjection
$(F^e_*\mathcal{F}_{2,e})_{\mathcal{I}_x}\rightarrow
\mathcal{F}_2\otimes\mathcal{I}_x/\mathcal{I}_z\cong k(x)$. The
required vanishing
$H^i(A,f_*(F^e_*\mathcal{F}_{2,e})_{\mathcal{I}_x}\otimes P)=0$ for
all $i>0$ and all $P\in \text{Pic}^0(A)$ is shown in Step 3.

\vspace{0.3cm}\emph{Step 7.} We show that for general $x\in X$, any
irreducible length two zero dimensional scheme $z$ with support $x$
and all $Q\in\text{Pic}^0(A)$, we can find a section in
$(F^e_*\mathcal{F}_{3,e})_{\mathcal{I}_x}\otimes f^*Q$ which induces
a section in $\mathcal{F}_3\otimes f^*P\otimes \mathcal{I}_x$ not
vanishing at $z$.

Let the kernel of the composition
$F^e_*\widetilde{\mathcal{F}_{3,e}}\rightarrow
\mathcal{F}_3\rightarrow \mathcal{F}_3\otimes k(x)$ be
$(F^e_*\widetilde{\mathcal{F}_{3,e}})_{\mathcal{I}_x}$. As in Step 2
and Step 4, near a general point $x$,
$F^e_*\widetilde{\mathcal{F}_{1,e}}^-$ induces homomorphisms from
$F^e_*\mathcal{F}_{2,e}$ to $F^e_*\widetilde{\mathcal{F}_{3,e}}$
which is commutative with the homomorphisms from $\mathcal{F}_2$ to
$\mathcal{F}_3$ induced by $\mathcal{O}_X(K_X)$. Hence
$F^e_*\widetilde{\mathcal{F}_{1,e}}^-$ induces homomorphisms between
the kernels of $F^e_*\mathcal{F}_{r,e}\rightarrow
\mathcal{F}_r\otimes k(x)$, i.e, from
$(F^e_*\mathcal{F}_{2,e})_{\mathcal{I}_x}$ to
$(F^e_*\widetilde{\mathcal{F}_{3,e}})_{\mathcal{I}_x}$. As the
homomorphism induced by $F^e_*\widetilde{\mathcal{F}_{1,e}}^-$ and
$\mathcal{O}_X(K_X)$ is commutative, we have the following
commutative diagram near a general point $x$.
$$\xymatrix{
(F^e_*\mathcal{F}_{2,e})_{\mathcal{I}_x} \ar[rr] \ar[d]& &
(F^e_*\widetilde{\mathcal{F}_{3,e}})_{\mathcal{I}_x} \ar[d] \\
\mathcal{F}_2\otimes \mathcal{I}_x/\mathcal{I}_z \ar[r] &
\mathcal{O}_X(K_X)\otimes \mathcal{F}_2\otimes
\mathcal{I}_x/\mathcal{I}_z \ar^-\simeq[r] & \mathcal{F}_3\otimes
\mathcal{I}_x/\mathcal{I}_z }$$

For any $Q\in\text{Pic}^0(A)$, we may pick $P\in\text{Pic}^0(A)$
such that $P$ and $Q\otimes P^\vee$ are both general. Similar to
Step 2 and Step 4, we have that
$$F^e_*\mathcal{F}_{1,e}^- \rightarrow
F^e_*\widetilde{\mathcal{F}_{1,e}}^- \rightarrow
\mathcal{H}om_{\mathcal{O}_X}(\mathcal{F}_2\otimes
\mathcal{I}_x/\mathcal{I}_z,\mathcal{F}_3\otimes
\mathcal{I}_x/\mathcal{I}_z) \cong k(x)$$ is surjective. We may
apply Corollary \ref{gfinite} and Remark \ref{var_gfinite} and get
that
$$H^0(X,F^e_*\mathcal{F}_{1,e}^-\otimes f^*(Q\otimes
P^\vee))\rightarrow \text{Hom}_{\mathcal{O}_X}(\mathcal{F}_2\otimes
f^*P\otimes \mathcal{I}_x/\mathcal{I}_z,\mathcal{F}_3\otimes
f^*Q\otimes \mathcal{I}_x/\mathcal{I}_z)$$ is surjective.

By Step 6, we have a section $s\in H^0(X,
F^e_*\mathcal{F}_{2,e}\otimes f^*P)$ restricting to 0 in
$\mathcal{F}_2\otimes f^*P\otimes \mathcal{O}_X/\mathcal{I}_x$ and
whose image in $\mathcal{F}_2\otimes f^*P\otimes
\mathcal{I}_x/\mathcal{I}_z$ does not vanish. By the discussion
above, we have a section $s^-\in
H^0(X,F^e_*\mathcal{F}_{1,e}^-\otimes f^*(Q\otimes P^\vee))$
inducing a nonzero homomorphism between $\mathcal{F}_2\otimes
f^*P\otimes \mathcal{I}_x/\mathcal{I}_z$ and $\mathcal{F}_3\otimes
f^*Q\otimes \mathcal{I}_x/\mathcal{I}_z$. Hence, $s^-\otimes s$
gives a section in $H^0(X, F^e_*\mathcal{F}_{3,e}\otimes f^*Q)$
restricting to 0 in $\mathcal{F}_3\otimes f^*Q\otimes
\mathcal{O}_X/\mathcal{I}_x$ and whose image in
$\mathcal{F}_3\otimes f^*Q\otimes \mathcal{I}_x/\mathcal{I}_z$ does
not vanish.

\vspace{0.3cm}\emph{Step 8.} By Step 7, for all
$Q\in\text{Pic}^0(A)$, we have a surjection
$$H^0(X,(F^e_*\mathcal{F}_{3,e})_{\mathcal{I}_x}\otimes f^*Q)\rightarrow
H^0(X,\mathcal{F}_3\otimes f^*Q\otimes
\mathcal{I}_x/\mathcal{I}_z).$$ Since this surjection factors
through $H^0(X,\mathcal{F}_3\otimes f^*Q\otimes \mathcal{I}_x)$, we
have that $\mathcal{F}_3\otimes f^*Q$ separates tangent vectors at
general points for all $Q\in\text{Pic}^0(A)$.

\vspace{0.3cm}Since
$\mathcal{F}_3=\mathcal{O}_X(4K_X)\otimes\tau(\|3\Delta\|)$ and
$\tau(\|3\Delta\|)$ is an ideal, we can conclude that $|4K_X|$
induces a birational map.

\end{proof}


\begin{thebibliography}{10}

\bibitem[BMS08] {BMS} M. Blickle, M. Musta\c{t}\u{a} and K. E.
Smith, \textit{Discreteness and rationality of F-thresholds},
Michigan Math. J. 57 (2008), 463-483.

\bibitem[CH02] {CH} J. A. Chen and C. Hacon, \textit{Linear series of
irregular varieties}, Proceedings of the symposium on Algebraic
Geometry in East Asia. World Scientific (2002), 143-153.

\bibitem[Hacon11] {H} C. Hacon, \textit{Singularities of pluri-theta
divisors in Char $p>0$}, preprint, arXiv:1112.2219.

\bibitem[HM06] {HM} C. Hacon and J. M$^\text{c}$Kernan, \textit{Boundedness of
pluricanonical maps of varieties of general type}, Invent. Math. 166
(2006), 1-25.

\bibitem[JLT11] {JLT} Z. Jiang, M. Lahoz and S. Tirabassi, \textit{On
the Iitaka fibration of varieties of maximal Albanese dimension},
preprint, arXiv:1111.6279.

\bibitem[Mukai81]{M81} S. Mukai, \textit{Duality between $D(X)$ and
$D(\hat{X})$ with its application to Picard sheaves}, Nagoya Math.
J. Vol. 81 (1981), 153-175.

\bibitem[Musta\c{t}\u{a}11] {M} M. Musta\c{t}\u{a}, \textit{The non-nef
locus in positive characteristic}, preprint, arXiv:1109.3825v1.

\bibitem[Schwede11] {S} K. Schwede, \textit{A canonical linear system
associated to adjoint divisors in characteristic $p>0$}, preprint,
arXiv:1107.3833v3.

\bibitem[Takayama06] {Ta} S. Takayama, \textit{Pluricanonical systems
on algebraic varieties of general type}, Invent. Math. 165 (2006), 551-587.

\bibitem[Tsuji06] {T06} H. Tsuji, \textit{Pluricanonical systems of
projective varieties of general type. I}, Osaka J. Math. 43 (2006),
no. 4, 967-995.

\bibitem[Tsuji07] {T07} H. Tsuji, \textit{Pluricanonical systems of
projective varieties of general type. II}, Osaka J. Math. 44 (2007),
no. 3, 723-764.

\end{thebibliography}
\end{document}